\documentclass[a4paper]{article}
\pdfoutput=1
\usepackage{amsmath,amsfonts,amsthm,amssymb,dsfont,color,graphicx,subfigure,caption,url}
\usepackage{subfig,multirow,booktabs,latexsym}  
\def\be{\begin{equation}}
\def\ee{\end{equation}}
\def\sign{\hbox{\rm sign}\,}
\def\abs#1{\left\vert #1 \right\vert} 
\newtheorem{lemm}{Lemma}[section]

\newtheorem{thm}{Theorem}[section]
\newtheorem{prop}{Proposition}[section]
\newtheorem{lemma}{Lemma}[section]

\newtheorem{algo}{Algorithm}[section]

\def\l1{L1}  %
\def\bp{\begin{prop}}
\def\ep{\end{prop}}
\def\bpf{\begin{proof}}
\def\epf{\end{proof}}
\def\bd{\begin{defi}}
\def\ed{\end{defi}}
\def\br{\begin{oss}}
\def\er{\end{oss}}
\def\bl{\begin{lemma}}
\def\el{\end{lemma}}

\title{A fast nonconvex Compressed Sensing algorithm for highly low-sampled MR images reconstruction}
\author{D. Lazzaro$^1$,  E. Loli Piccolomini$^1$ and F. Zama$^1$ \\
\normalsize
$^1$      
 \normalsize
 Department of Mathematics, University of Bologna,\\
} \date{\today}
\begin{document}
\maketitle
\begin{abstract}
In this paper we present a fast and efficient method for the reconstruction of Magnetic Resonance Images (MRI) from 
severely under-sampled data. 
From the Compressed Sensing theory we have  mathematically modeled the problem as a constrained minimization problem
with a family of  non-convex regularizing objective functions depending on a parameter  and a  least squares data fit constraint. 
We  propose a fast and efficient algorithm, named Fast NonConvex Reweighting (FNCR) algorithm,  
based on an iterative scheme where the non-convex problem is approximated by its convex 
linearization and the penalization parameter is automatically updated. 
The convex problem is solved by a Forward-Backward  procedure, where the Backward step is performed 
by a Split Bregman strategy. Moreover, we propose a new efficient iterative solver for the arising linear systems. 
We prove the convergence of the proposed FNCR method.
The results on synthetic phantoms and real images show that the
algorithm is  very well performing  and computationally efficient, even when compared to the best performing methods proposed in the literature.
\end{abstract}
%
\section{Introduction \label{intro}}
The Compressed Sensing, or Compressive Sampling (CS),  is a quite recent theory (the first publications \cite{Candes2006,Donoho2006}
appeared in 2006) 
aiming at recovering signals and images from fewer measurements than those required by the traditional Nyquist law, 
under specific requirements on the data. CS has received emerging attention in medical imaging and 
in Magnetic Resonance Imaging (MRI) in particular, 
where the assumptions behind the theory are satisfied. Reduced sampled acquisitions are employed both in dynamic MRI 
\cite{Gamper2008,MAJUMDAR2013,Landi2008}, to speed up the acquisition  
to better catch the dynamic process inside the body, and in the single MRI, where fast imaging is essential 
to improve patient care and reduce the costs \cite{Lustig2007}. 
The aim is to reconstruct  MR images from  highly  
under-sampled data  and the research in this field seeks for methods that further reduce the amount of 
acquired data without degrading the reconstructed images.  
Some important papers, such as \cite{Lustig2007,Wang2016,chartrand2009,trzasko2009}, studied the theoretical application of CS 
in the reconstruction of  MR images from low-sampled data. 

 CS theory assumes that the signal or the image  is sparse in some transform domain: in the case of MR the image can be sparse, 
 for example,  in the wavelet domain or in the gradient domain.
 In this paper we assume that the image  is sparse in the gradient domain.
 Moreover CS requires a
 nonlinear reconstruction process enforcing both the sparsity in the transform 
domain and the data consistency. In particular, the nonlinear reconstruction process is usually modeled as 
the minimization of a function imposing the sparsity on the coefficients in the transform domain under 
a constraint on data consistency. In this paper we consider the following general nonlinear  model for the MRI reconstruction: 
\begin{equation}
\min_u F(Du) \ \ 
s.t. \ \mathcal{J}(u)<\epsilon
\label{eq1}
\end{equation}
where $u$ is the  image to be recovered,  $F$ is the sparsifying transform in the gradient domain $D$,  $\mathcal{J}(u)$ 
is the data-fit function and $\epsilon$ is a parameter controlling the fidelity of the reconstruction to the measured data.
The problem \eqref{eq1} is often solved in its penalized formulation:
\begin{equation}
\min_u \mathcal{J}(u) +\lambda F(Du)
\label{eq2}
\end{equation}
where $\lambda>0$ is called penalization or regularization parameter.

The choice of the sparsifying function $F$ is crucial for the effectiveness of the application. 
In literature, there are different proposals for the choice of $F$.
It is well known that the $\ell_0$ semi-norm, measuring the cardinality of its argument, 
is the best possible sparsifying function, but its minimization is computationally very expensive. 
A remarkable result of the CS theory \cite{Candes2006,Donoho2006} states 
that the signal recovery is still possible if one substitutes the $\ell_1$ norm to the $\ell_0$ semi-norm under suitable hypotheses;
hence some authors considered the sparsifying function  as the Total Variation  function \cite{Lustig2007, Zhang2010}. 

In order to better approximate the $\ell_0$ semi-norm, Chartrand in \cite{chartrand2007,chartrand2009} used the $\ell_p$, $ 0 <p <1$,
nonconvex sparsifying norm. Even if the global minimum of \eqref{eq1} cannot be guaranteed in this case, the author presents
a proof of asymptotic convergence of $\ell_p$ towards $\ell_0$ by suitably setting the restricted isometry constants.
In \cite{trzasko2009} the authors generalize this idea by using  a class of sparsifying functions homotopic with the $\ell_0$ semi-norm.

In \cite{YIFEI} the authors choose as function $F$ the difference between the convex norms  $\ell_1$ and  $\ell_2$  and they 
 propose a Difference of Convex functions Algorithm (DCA), proving that the DCA approach converges to stationary points. 
In practice, the DCA iterations, when properly stopped, are often close to the global
minimum and produce very good results. 
In \cite{IL1MRI}, the authors generalize the proposal presented in \cite{YIFEI}, by minimizing a class
of concave sparse metrics in a general DCA-based CS framework. The sparsifying function can be written as:
\begin{equation}
F(Du)= \sum_i p_i(Du),
\end{equation}
where, for each $i$, $p_i$, is defined on $[0, +\infty)$ and is concave and increasing.

In \cite{DEF2016} the authors used a family of sparsifying nonconvex functions of the gradient of $u$, $F_{\mu}(Du)$, 
homotopic with the $\ell_0$ semi-norm, depending on a parameter $\mu$ and converging to $\ell_0$ as $\mu$ tends to zero. 
In that paper the algorithm for solving the resulting nonconvex minimization problem had been sketched together with a simple test problem 
of MRI reconstruction from low sampled data that produced very promising results.

Aim of this paper is to present in details the 
Fast Non Convex Reweighting (FNCR) algorithm, to prove its convergence and to show the results on a wide experimentation of MRI reconstructions.
Moreover, we propose a new matrix splitting strategy which allows us to obtain a very efficient explicit 
iterative solver for the linear systems solution representing most time consuming step of FNCR.

In the FNCR algorithm the minimization problem \eqref{eq1} is solved by an iterative scheme where the parameter $\mu$ gradually approaches zero and each nonconvex problem  is approximated with a convex problem by a weighted linearization.
The resulting convex  problem is solved by an iterative Forward-Backward (FB) algorithm, and a Weighted Split-Bregman iteration is applied in the 
Backward step.   
The implementation of the Weighted Split-Bregman iteration requires the solution of linear systems, 
representing the  most computationally expensive step of the whole algorithm. 
By exploiting a decomposition of the Laplacian matrix, we propose a new splitting that exhibits further filtering properties producing very good results  both in terms  
of computational efficiency and image quality.

We present many numerical experiments of MRI reconstructions from reduced sampling, both on phantoms and real images, with different low-sampling masks  and we compare 
our algorithm with some of the most recent software  proposed in the literature  for CS MRI. 
The results show that the FNCR algorithm  is  very well performing  
and computationally efficient especially when very high under-sampling is assumed.

The paper is organized as follows: in section \ref{model} we motivate the choice of the nonconvex model; 
in section \ref{FNCR} we explain in details the proposed FNCR algorithm
for the solution of the nonconvex minimization problem;  in section \ref{numres}
we show the results obtained in the numerical experiments and at last in section \ref{concl} we report some conclusions. 
The Appendices A and B contain the proofs of the proposed  convergence theorems.
 \section{The nonconvex regularization model\label{model}}
In this section we discuss the assumed mathematical model  \eqref{eq1} for the reconstruction of MR images from low sampled data.

Under-sampled MR data are represented by a reduced set of acquisitions in the Fourier domain (K-space), affected by a dominant gaussian noise.  
Hence the data fidelity function in 
\eqref{eq1} is chosen as the least squares function:
\begin{equation} \mathcal{J}(u) \equiv \frac{1}{2} \|\Phi u -z\|_2^2
\label{LS}
\end{equation}
where $z$ is the vector of the acquired data in the Fourier space, $u$ is the image to be reconstructed and 
$\Phi$ is the under-sampling Fourier matrix,  obtained by the Hadamard  product between  the full resolution Fourier matrix $\mathbf{F}$ 
and the under-sampling mask $\mathcal{M}$, i.e.
\begin{equation}\Phi = \mathcal{M} \circ \mathbf{F}.
\label{eq:mask}
\end{equation}
 
When the MR K-space is only partially measured, the inversion problem to obtain the MR image is under-determined, 
hence it has infinite possible solutions.
Following the CS theory, a way to choose one of these infinite solutions is to impose a prior in the sparsity solution domain. 
As we discussed in the introduction, in this paper we consider the gradient image domain, $Du$, as the sparsity domain. 

Concerning the choice of the sparsifying function $F$,  we follow the approach proposed by Trzasko et. al in \cite{trzasko2009}. 
Starting from the proposal of Chartrand in \cite{chartrand2007} 
of using $\ell_{p}$ functions as priors ($0<p<1$), 
they show that the reconstructed signals are better that those obtained with $\ell_1$ prior, even if the $\ell_p$  
family is constituted by nonconvex functions with  many possible local minima.
Moreover, as $p$ approaches to zero, the number of 
data necessary to accurately reconstruct a signal decreases. Instead of the $\ell_p$ priors, 
the authors in \cite{trzasko2009} propose to use  classes of functions, 
depending on a scale parameter and providing a refined measure
of the signal sparsity as the scale parameter  approaches to zero.
Following the same idea we define a class of nonconvex functions $F_{\mu}(Du)$,
 depending on the parameter $\mu >0 $, satisfying the  property:
\begin{equation}
\lim_{\mu \to 0}{\sum_{\Omega} F_{\mu}(Du)=\|Du\|_0}.
\label{eq:homot}
\end{equation}
where $\Omega$ is the image domain.
In our paper we consider the following  family of  functions $F_{\mu}(Du)$:
\begin{equation}
 F_{\mu}(Du)=\sum_{i,j} \left ( \psi_{\mu}(|u_x|_{i,j})+ \psi_{\mu}(|u_y|_{i,j}) \right )
\label{eq:functions}
\end{equation}
where $Du=(u_x,u_y)  \in \mathbb{R}^{N \times 2N}$ is the discrete gradient of the image $u \in \mathbb{R}^{N \times N}$, 
whose components $u_x$ and  $u_y$ are obtained 
 by backward finite differences approximations of the
 partial derivatives in the $x$ and $y$ spatial coordinates. We choose the function 
$ \psi_{\mu}(t) : \mathbb{R} \setminus \{0\} \rightarrow \mathbb{R}$ having the following expression:
 \begin{equation}
 \psi_{\mu}(t)=\frac {1 }{\log(2)} \log   \left ( \frac {2 }{1+e^{-\frac {|t| }{\mu}}   } \right ) ,\quad \mu>0
 \label{Function_PSI}
 \end{equation}
that has been shown in \cite{Montefusco2013} to be  very well performing in a CS framework.
In the same paper the authors show that $\psi_{\mu}(t)$ has the property
$$ \lim_{\mu \rightarrow 0} \psi_{\mu}(t)  = \| t \|_0. $$
and that it is an increasing, nonconvex, symmetric, twice differentiable  function.

Hence, we solve a sequence of constrained minimization problems 
\begin{equation}
\arg\min_u F_{\mu}(Du) \ \ \ s.t. \ \|\Phi u-z\|_2^2<\epsilon
\label{eq:constr}
\end{equation}
with $\mu$ varying  decreasing to zero.
\section{The  Fast Non Convex Reweighting Algorithm (FNCR) \label{FNCR}}

In this section we show the details of the FNCR method  
for the solution of the sequence of minimization problems \eqref{eq:constr}, combining 
different  suggestions proposed in 
 \cite{Montefusco2012, Montefusco2013} and recently in \cite{Lazzaro2017,DEF2016}.

FNCR is constituted by five nested iterations.
 In order to give a clear and simple explanation   we split the FNCR description into four  paragraphs.
 In the description each loop index  is represented  by a different  alphabetical symbol.
 \begin{itemize}
  \item \ref{sub1} The continuation scheme (index $\ell$) and the Iterative Reweighting method ( index $h$).
  \item \ref{FBstep} The Forward-Backward (FB) strategy for the solution of the convex minimization problem obtained in  \ref{sub1} (index $n$).
  \item \ref{SBstep} The Weighted Split Bregman algorithm for the solution of the Backward step in \ref{FBstep} (index $j$) .
	\item \ref{LS} The new iterative scheme for the solution of the linear system arising in \ref{SBstep} (index $m$).
 \end{itemize}
Finally, collecting all the details, we present the steps of the FNCR method  in Algorithm \ref{alg2}.
We report all the  proofs of the theorems in appendices \ref{A} and \ref{A_B}.

\subsection{The continuation scheme and the Iterative Reweighting l1 method \label{sub1}}

The sequence of problems defined by \eqref{eq:constr} is implemented by an iterative continuation scheme on $\mu$. 
 Given a starting vector $u^{(0)}=\Phi^T z$ a sequence of  approximate solutions $u^{(1)},u^{(2)}, \ldots u^{(\ell+1)}$ 
 is computed as:
\begin{equation}
u^{(\ell+1)}=\arg\min_u  {F_{\mu_{\ell}}(Du)} \ \ \ s.t. \  \|\Phi u-z\|_2^2<\epsilon, \ \ \ell=0,1 \ldots
\label{eq:continuation}
\end{equation}
with $ \mu_{\ell} <  \mu_{\ell-1}$. In our work  the parameter $\mu$ is decreased with the following rule:
\begin{equation}
\mu_{\ell}= 0.8 \mu_{\ell-1}.
\label{eq:mu}
\end{equation} 

At each iteration $\ell$, the constrained minimization problem \eqref{eq:continuation} is stated  in its  penalized formulation:
\begin{equation}
u^{(\ell+1)}=\arg\min_u \left \{ \lambda F_{\mu_{\ell}}(Du) + \frac{1}{2} \| \Phi u - z \|_2^2 \right \}, \ \ \ \ell=0,1, \ldots.
\label{eq:nconv}
\end{equation}

An Iterative Reweighting $l_1$ ($IRl1$) algorithm \cite{Candes2008} is applied to each  
nonconvex minimization subproblem $\ell$ in \eqref{eq:nconv}. 
The $IRl1$ method  
 computes a sequence of approximate solutions $\bar u^{(0)}, \bar u^{(1)}, \ldots, \bar u^{(h)}$,
 obtained from convex minimizations. In appendix A we prove the convergence of the sequence $\{ \bar u^{(h)} \}$
  to the solution of the original nonconvex problem $u^{(\ell+1)}$.
In particular, at each iteration $h$ of the $IRl1$ algorithm, the nonconvex function $F_{\mu_{\ell}}(Du)$  is locally approximated 
by its convex majorizer   $\mathcal{F}_{h,\mu_{\ell}}(Du)$ obtained by linearizing $F_{\mu_{\ell}}(Du)$
at the gradient of  the  approximate solution  $\bar u^{(h)}$:
\begin{equation}
\mathcal{F}_{h,\mu_{\ell}}(Du) = F_{\mu_{\ell}}(D\bar u^{(h)}) + f_{h,\mu_{\ell}}(Du)-f_{h,\mu_{\ell}}(D\bar u^{(h)})
\label{eq:fbig}
\end{equation}
where
\begin{equation}
f_{h,\mu_{\ell}}(Du)= \sum_{i,j=1}^{N} \left ( \psi'_{\mu_{\ell}}(|\bar u_x^{(h)}|_{i,j}) | u_x|_{i,j} + 
\psi'_{\mu_{\ell}}(|\bar u_y^{(h)}|_{i,j}) | u_y|_{i,j}  \right )
\label{fell}
\end{equation}
($\psi(t)$ is defined by \eqref{Function_PSI}).
Hence, at each iteration $h$ of the $IRl1$ algorithm we solve the convex minimization problem:
\begin{equation}
\bar u^{(h+1)}= \arg \min_{u} {\cal P}(\lambda_{h},\mathcal{F}_{h,\mu_{\ell}}(Du),u)
 \ \ \ h =0, 1, \ldots
\label{IRL1}
\end{equation}
where 
\begin{equation}
\label{eq:defP}
{\cal P}(\lambda_{h},\mathcal{F}_{h,\mu_{\ell}}(Du),u)=\left \{\lambda_{h} \mathcal{F}_{h,\mu_{\ell}}(Du)  +\frac{1}{2}||\Phi u- z||_2^2 \right \}  .
\end{equation}
Since $f_{h,\mu_{\ell}}(Du)$ is the only term of \eqref{eq:fbig} depending on $u$, the minimization problem \eqref{IRL1} reduces to the following:
\begin{equation}
\bar u^{(h+1)}= \arg \min_{u}  \left \{\lambda_{h} f_{h,\mu_{\ell}}(Du)  +\frac{1}{2}||\Phi u- z||_2^2 \right \}.
\label{eq:fsmall}
\end{equation}

Concerning the computation of the regularization  parameters $\lambda_{h}$, 
starting from a sufficiently large initial estimate $\lambda_0 = r_0 \| u^{(0)} \|_1$,
where  $0<r_0<1$ is an assigned constant
we perform an iterative update creating a decreasing sequence $\{ \lambda_{h}\}, h=0,1, \ldots$ as:
\begin{equation}
\lambda_{h+1}=\lambda_{h} \frac{ {\cal P}(\lambda_{h},\mathcal{F}_{h,\mu_{\ell}}(D\bar u^{(h)}),\bar u^{(h)})}{{\cal P} (\lambda_{h-1},\mathcal{F}_{h-1,\mu_{\ell}}(D\bar u^{(h-1)}),\bar u^{(h-1)})}.
\label{eq:lambda}
\end{equation}
In fact we observe that if 
$\lambda_{h}<\lambda_{h-1}$ then it easily follows \cite{Montefusco2012} that:
$${\cal P}(\lambda_{h},\mathcal{F}_{h,\mu_{\ell}}(D\bar u^{(h)}),\bar u^{(h)})< {\cal P}(\lambda_{h-1},\mathcal{F}_{h-1,\mu_{\ell}}(D\bar u^{(h-1)}),\bar u^{(h-1)}).$$
We have heuristically tested that this adaptive update of $\lambda_{\ell}$ reduces the number of  iterations
for the solution of the inner minimization problem.
A warm starting rule is adopted to compute the initial iterate of the $IRl1$ scheme.

The outline of the scheme  is shown in  algorithm  \ref{alg1} that is stopped when the relative distance between two successive iterates is small enough.

\begin{table}[h]
\begin{algo}
\begin{center}
\begin{tabular}{c}
\hline
 \textbf{Continuation scheme with IRl1 method}\quad \quad \quad \quad\quad \quad \quad \quad \quad \quad \quad \quad   \\
 \hline 
\begin{minipage}[c]{11cm}
\begin{tabular}{ll}
(Input: z, $r_0$) \\
$u^{(0)}=\phi^T z, \, \lambda_0=r_0 \|u^{(0)}\|_1$, $\mu=\|\nabla u^{(0)} \|_1$ \\
$w_{i,j}^x=1$,\, $w_{i,j}^y=1$, $ \forall i,j$\\
$\bar u^{(0)}=u^{(0)}$, $\ell=0$ \\[1pt]
{\tt repeat}  \ \ \  \ \ \\[2pt]
\hspace{0.5cm} $h=0$ \\
\hspace{0.5cm} {\tt repeat}  \ \ \ ($IRl1$ iteration)  \ \ \\[2pt]
%
\quad
  \vbox{\begin{multline} 
 \bar u^{(h+1)}= \arg \min_{u}  \left \{\lambda_{h} \sum_{i,j} \left ( w_{i,j}^{x} | u_x|_{i,j} + w_{i,j}^{y} |u_y|_{i,j}  \right )+  \right. \\
     \left . +\frac{1}{2}||\Phi u- z||_2^2 \right \}
  \label{eq:alg1}
\end{multline}}\\

 \quad \vbox {\begin{equation} \hspace{-2.8cm} w_{i,j}^{x}=\psi'_{\mu_\ell}(|\bar u_x^{(h+1)}|_{i,j}), w_{i,j}^y=\psi'_{\mu_\ell}(|\bar u_y^{(h+1)}|_{i,j}) \label{eq:weights1}
\end{equation}
}, \quad \\[2pt]
\quad \quad Update $\lambda_{h+1}$ as in \eqref{eq:lambda}\\ 
\quad \quad $h=h+1$ \\
\hspace{0.5cm}{\tt until}  convergence \\
$u^{(\ell+1)}=\bar u^{(h)}$\\
$\lambda_0=\lambda_{h}$ \\
$\ell=\ell+1$,   \\[2pt]
Update $\mu_{\ell}$ as in \eqref{eq:mu} \\[2pt]
{\tt until}  convergence 
\end{tabular}
\end{minipage}\\ \hline
\end{tabular}
\end{center}
 \label{alg1}
\end{algo} 
\end{table}

\subsection{The Forward-Backward strategy for the solution of  problem \label{FBstep}}

In this paragraph we analyze the details of the method used in FNCR 
for the solution of the convex minimization problem \eqref{eq:fsmall}.

Since $\psi'_{\mu}(\cdot)>0$ when $\mu >0$,
we rewrite \eqref{fell} as follows:
$$ f_{h,\mu_{\ell}}(Du) = \|\nabla_x^w u\|_1+\|\nabla_y^w u\|_1$$
 where

$$\|  \nabla_x^w u\|_1=\sum_{i,j=1}^N  (w_{i,j}^x) | u_x|_{i,j}, \ \ \|  \nabla_y^w u\|_1=\sum_{i,j=1}^N   (w_{i,j}^y) | u_y|_{i,j}$$
with
$w_{i,j}^x$ and $w_{i,j}^y$ defined as in \eqref{eq:weights1}.

Hence problem  \eqref{eq:fsmall} is rewritten as:
\begin{equation}
{\bar u}^{(h+1)}= \arg \min_{u} \left \{ \frac{1}{2 }\| \Phi u - z \|_2^2+ \lambda_{h} \left(\|\nabla_x^w u\|_1+\|\nabla_y^w u\|_1 \right) \right \}
\label{eq:convex}
\end{equation}
The convex minimization problem \eqref{eq:convex} has a unique solution as proved in \cite{Lazzaro2017} 
and we compute it  by a converging sequence of 
Accelerated Forward-Backward steps $(\hat{v}^{(n)}, \hat{u}^{(n)}), n=1,2, \ldots$ where a FISTA  acceleration strategy \cite{Teboulle2009} 
is applied to the backward step:
\begin{flalign}
\hat{v}^{(n)}={\hat u}^{(n-1)} +\beta \Phi^t \left (z - \Phi{\hat u}^{(n-1)} \right ) \\
{\tilde u}^{(n)} = \arg\min_{u} \left \{\lambda_{h}  \left (\|  \nabla_x^w u\|_1+ \|  \nabla_y^w u\|_1 \right ) +\frac{1}{2  \beta} 
\| u - \hat{v}^{(n)} \|_2^2 \right \}\label{B} 
\end{flalign}

 \begin{equation}
\hat u^{(n)}={\tilde u}^{(n-1)}+\alpha({\tilde u}^{(n)}-{\tilde u}^{(n-1)}), 
 \label{eq:F}
 \end{equation}
and $\alpha$ is chosen as follows:
 \begin{equation}
t_{n}=\frac{1+\sqrt{1+4t_{n-1}^2}}{2}, \ \ \alpha=\frac{t_{n-1}-1}{t_{n}}, \ \ t_0=1 
\label{eq:alpha}
\end{equation}
In order to ensure the  convergence of the sequence $(\hat{v}^{(n)}, \hat{u}^{(n)})$ to the solution of \eqref{eq:convex}, the following condition on $\beta$ must hold \cite{Combettes2005}:
$$ 0 < \beta < \frac{2}{\lambda_{max}(\Phi^t\Phi)}$$ 
where $\lambda_{max}$ is the maximum eigenvalue in modulus.
In our MRI  application the matrix  $\Phi$ is orthogonal, hence the condition is $ 0 < \beta < 2$.
We observe that while $\hat{v}^{(n)}$ and  $\hat u^{(n)}$ are computed by explicit formulae, the computation of ${\tilde u}^{(n)}$
requires a more expensive method that we describe in the next paragraph. 

The Forward-Backward iterations are stopped with  the following  stopping condition:
\begin{equation}
\Delta_{n}-\Delta_{n-1}  < \gamma \lambda_{h}  \ \ \ \hbox{where} \ \ \ 
\Delta_n = \sum_{i,j=1}^N  \left ( (w_{i,j}^x) | \hat u^{(n)}_x|_{i,j}+ (w_{i,j}^y)|\hat u^{(n)}_y|_{i,j}\right). .
\label{stop}
\end{equation}
and $\gamma >0$ is a suitable tolerance.

\subsection{The Weighted Split Bregman algorithm for the solution of the Backward step \label{SBstep}}

In this  paragraph we show how 
the minimization problem \eqref{B} can be  efficiently solved by means of
a splitting variable strategy, proposed in the Weighted Split Bregman algorithm \cite{Zhang2010}.
Introducing two auxiliary vectors $D_x, D_y \in \mathbb{R}^{N^2}$ we rewrite \eqref{B} as 
a constrained minimization problem as follows:
\begin{equation}
\label{SB}
 \min_{u} \left \{\frac{1}{2 \beta }\| u - \hat{v}^{(n)} \|_2^2 +  \lambda_{h} \left(\| D_x\|_1 + \| D_y \|_1 \right) \right \}, 
 \ \ \ s.t. \ \ D_x = \nabla_x^w u, \ \ D_y = \nabla_y^w u,
\end{equation}
which can be stated in its quadratic penalized form as:
\begin{equation}
\min_{u, D_x, D_y} \left \{\frac{1}{2 \beta } \| u - \hat{v}^{(n)} \|_2^2 +  \lambda_{h} \left (\| D_x\|_1 + \| D_y \|_1 \right)  + 
\frac{\theta}{2}
\left (  \|  D_x-\nabla_x^w u \|_2^2 +  \| D_y-\nabla_y^w u \|_2^2 \right ) \right \}
\label{eq:LSB}
\end{equation}
where $ \theta>0$ represents the penalty parameter.
In order to  simplify the notation, exploiting the symmetry in the $x$ and $y$ variables, 
we use the subscript $q$  indicating either $x$ or $y$.

By applying the Split Bregman iterations, given an initial iterate $U^{(0)}$, we compute a sequence $U^{(1)},U^{(2)}, \ldots, U^{(j+1)}$ 
by splitting \eqref{eq:LSB}  into three minimization problems as follows. 

Given $e_q^{(0)}=0$, $D_q^{(0)}=0$ and $U^{(0)}={\hat v}^{(n)}$, compute:


\begin{multline}
\label{UK}
 U^{(j+1)}=\arg \min_{u }\left\{ \frac{1}{2 \beta} \| u - \hat {v}^{(n)} \|_2^2+ 
 \frac {\theta}{2}  \| D^{(j)}_x-\nabla_x^w u -e_x^{(j)}\|_2^2+ \right . \\
\left. \frac { \theta}{2} \| D_y^{(j)}-\nabla_y^w u -e_y^{(j)}\|_2^2\right \}
\end{multline}
\begin{multline}
\label{Dx}
D_{q}^{(j+1)}=\arg \min_{ D_{q} }\left\{\lambda_{h}\|D_{q}\|_1+\frac {\theta}{2}  \| D_{q}-\nabla_{q}^w U^{(j+1)}-e_{q}^{(j)} \|_2^2 \right \}\\
= \mbox{Soft}_{\Lambda} (\nabla_{q}^w U^{(j+1)}+e_{q}^{(j)}),
\end{multline}
where 
\begin{equation}
\Lambda = \frac{\lambda_{h}}{\theta}
\label{eq:Lambda}
\end{equation}
and  $e_{q}^j$ is updated according to the following equation:
\begin{multline}
\label{ex}
e_{q}^{(j+1)}=e_{q}^{(j)}+\nabla_{q}^w U^{(j+1)}-D_{q}^{(j+1)}=e_{q}^{(j)}+\nabla_{q}^w U^{(j+1)}-\\ 
\mbox{Soft}_{\Lambda} (\nabla_{q}^w U^{(j+1)}+e_{q}^{(j)}) =\mbox{Cut}_{\Lambda} (\nabla_{q}^w U^{(j+1)}+e_{q}^{(j)}),
\end{multline}
 We remind that the Soft and the Cut operators apply point-wise  respectively as:
\begin{equation}
\label{Soft}
\mbox{Soft}_{\Lambda}(z)=\sign(z) \max \left\{ \abs{z}-{\Lambda},0\right \} 
\end{equation}
\begin{equation}
\label{Cut}
\mbox{Cut}_{\Lambda}(z)=z-\mbox{Soft}_{\Lambda} (z)= 
\left\{
\begin{tabular}{cc}
$\Lambda$&$z>{\Lambda}$\\
$z$&$-{\Lambda}\leq z \leq {\Lambda}$\\
$-{\Lambda}$&$z< -{\Lambda}$.\\
\end{tabular}
\right.
\end{equation}

By imposing first order optimality conditions in \eqref{UK},we compute the minimum $U^{(j+1)}$  by solving the following linear system
\begin{multline}
\label{SistL}
(\frac{1}{\beta}I-\theta \Delta^w)U^{(j+1)}= \frac{1}{\beta}\hat{v}^{(n)} + \theta  (\nabla_x^w)^T(D_x^{(j)}-e_x^{(j)})+\\
 \theta (\nabla_y^w)^T(D_y^{(j)}-e_y^{(j)})
\end{multline}

where 
\begin{equation}
\Delta^w=- \left ((\nabla_x^w)^T \nabla_x^w+(\nabla_y^w)^T \nabla_y^w \right ).
\label{eq:lapl}
\end{equation}

Defining: 
\begin{equation}
A=(I-\beta \theta \Delta^w)
\label{eq:matr}
\end{equation}
and
\begin{equation}
b^{(j)}=\hat{v}^{(n)} + \beta \theta  (\nabla_x^w)^T(D_x^{(j)}-e_x^{(j)})+\\
\beta \theta (\nabla_y^w)^T(D_y^{(j)}-e_y^{(j)})
 \label{eq:zj}
\end{equation}
the linear system \eqref{SistL} can be written as:
 \begin{equation}
 A U^{(j+1)} = b^{(j)},   
 \label{Slin1}
\end{equation}
Exploiting the structure of the matrix $A$ we obtain a matrix splitting of the form $E-F$ where $E$ is the Identity matrix
and $F$ is $\beta \theta \Delta^w$. We can prove that the iterative method, based on such a splitting, is convergent if 
\begin{equation}
 0 < \theta < \frac{1}{\beta \|\Delta^w \|. }
 \label{Theta}
\end{equation}
We report in Algorithm \ref{alg3} the details of its implementation.
\subsection{An efficient iterative method for the solution the linear systems \label{LS}}
In this paragraph we analyze the iterative method obtained by the matrix splittiong suggested by the structure of the matrix $A$ 
\eqref{eq:matr}and prove its convergence.
\begin{thm}
\label{th:t1}
 Let $E$ and $F$ define a splitting of the matrix $A=E-F$ in  \eqref{Slin1} as:
 \begin{equation}E = I, \ \ \ F= \beta  \theta \Delta^w.
\label{eq:split}
\end{equation}
  By choosing $\theta$ as in \eqref{Theta} we can prove that the spectral radius $\rho(E^{-1}F)<1$ and, for each right-hand side $B$, the following iterative method 
\begin{equation}
\label{solS}
X^{(m+1)}=F X^{(m)}+ B, \ \ m=0, 1, \ldots 
\end{equation}
 converges to the solution of the linear system $AX=B$.
\end{thm}
\begin{proof}
See Appendix \ref{A_B}.
\end{proof}

Hence we compute the Weighted Split Bregman solution $U^{(j+1)}, j=0,1,\ldots$ by means of the iterative method defined in \eqref{solS}
with $B = b^{(j)}$ as in \eqref{eq:zj}.

Substituting \eqref{eq:lapl} in \eqref{solS} we have:
\begin{equation}
X^{(m+1)}= - \beta  \theta  \left ((\nabla_x^w)^T \nabla_x^w+(\nabla_y^w)^T \nabla_y^w \right )X^{(m)}+ b^{(j)}.
\label{eq:eq_X}
\end{equation}
 By  substituting  \eqref{eq:zj} in \eqref{eq:eq_X} and  collecting $\nabla_x^w$, $\nabla_y^w$ we obtain :
\begin{multline}
 X^{(m+1)}= \hat{v}^{(n)} + \beta \theta \left [ (\nabla_x^w)^T(-\nabla_x^w X^{(m)} + D_x^{(j)}-e_x^{(j)})+ \right . \\
 \left . +(\nabla_y^w)^T(-\nabla_x^w X^{(m)}+ D_y^{(j)}-e_y^{(j)}) \right ]\label{eq:iterations}
\end{multline}

\subsubsection{Implementation notes \label{impl}} 

We further make some simple algebraic handlings with the aim of avoiding the explicit computation of 
$D_x^{(j)}$ and $D_y^{(j)}$ in \eqref{eq:iterations}. 

Subtracting \eqref{Dx} to \eqref{ex}  and using relation \eqref{Cut} we deduce that
$$\mbox{Soft}_{\Lambda} (\nabla_{q}^w U^{(j+1)}+e_{q}^{(j)})= (\nabla_{q}^w U^{(j+1)}+e_{q}^{(j)})-\mbox{Cut}_{\Lambda} 
\left(\nabla_{q}^w U^{(j+1)}+e_{q}^{(j)}\right).$$
Immediately it follows that
\begin{equation}
\label{Ottx}
e_{q}^{(j+1)}-D_{q}^{(j+1)}=2 e_{q}^{(j+1)}- (\nabla_{q}^w U^{(j+1)}+e_{q}^{(j)})
\end{equation}
Using \eqref{Ottx} in \eqref{eq:iterations}  we 
can  define  a more efficient update formula:
\begin{multline}
 X^{(m+1)}= \hat{v}^{(n)} - \beta \theta \left [ (\nabla_x^w)^T(\nabla_x^w X^{(m)} +2e_x^{(j)}-(\nabla_x^w U^{(j)}+e_x^{(j-1)})) \right .\\
 \left . +(\nabla_y^w)^T(\nabla_x^w X^{(m)}+ 2e_y^{(j)}-(\nabla_y^w U^{(j)}+e_y^{(j-1)})) \right ]
\label{eq:it_new}
\end{multline}

In Algorithm \ref{alg3} we report the function (fast\_split) for the solution of the Backward step \eqref{B}.
The  output variable  $U^{(j)}$ is the computed solution and $\bar m$ is the number of total performed iterations.
\begin{table}[!htbp]
\begin{algo}
\begin{center}
\begin{tabular}{c}\hline
\textbf{[$U^{(j)}$, $\bar m$]=fast\_split($\lambda, \theta, \beta,{\hat v}^{(n)},w^x,w^y$)}\\
 \hline 
\begin{minipage}[c]{11cm}
\begin{tabular}{ll}
$\Lambda=\frac{\lambda}{\theta}$ , $\bar m=0$, \\
 $U^{(0)}={\hat v}^{(n)}$, $e_x^{(0)}=e_y^{(0)}=0$;\\
\quad  $j=1$\\
\quad \quad {\it repeat } &\\
\quad \quad \quad $U_x={\nabla_x^w} U^{(j-1)}$; $U_y={\nabla_y^w} U^{(j-1)}$; & \\
\quad \quad \quad $X_x^{(0)}=U_x$; $X_y^{(0)}=U_y$ & \\
 \quad  \quad \quad$z_x= U_x+e_x^{(j-1)}$; $z_y= U_y+e_y^{(j-1)}$; &\\
 \quad \quad \quad $e_x^{(j)}=Cut_{\Lambda}(z_x)$; $e_y^{(j)}=Cut_{\Lambda}(z_y)$; &\\
 \quad\quad \quad $m=0$\\
\quad \quad \quad {\it (Solution of problem \eqref{Slin1}) }\\
\quad \quad \quad {\it repeat } &\\
 \quad \quad \quad \quad $X^{(m+1)}={\hat v}^{(n)}- \beta \theta\left({\nabla_x^w}^T(X_x+2 \cdot e_x^{(j)}-z_x)+
{\nabla_y^w}^T (X_y +2 \cdot e_y^{(j)}-z_y)\right)$&\\
 \quad   \quad \quad \quad $X_x={\nabla_x^w} X^{(m+1)}$; $X_y={\nabla_y^w} X^{(m+1)}$&\\
 \quad  \quad \quad \quad  $m=m+1$&\\
 \quad  \quad \quad {\it until {\tt  stopping condition \eqref{eq:stop_sb}}}&\\
 \quad \quad \quad$U^{(j+1)}=X^{(m)}$;&\\
 \quad \quad \quad $\bar m=\bar m+m;j=j+1$&\\
\quad \quad {\it until {\tt stopping condition \eqref{eq:stop_sb}}}&\\
\end{tabular}
\end{minipage} \\ \hline
\end{tabular}
\end{center}
 \label{alg3}
\end{algo} 
\caption{ fast\_split Algorithm for the solution of problem \eqref{B}}
\end{table}

The  stopping condition of both the loops (with index $j$ and $m$) is defined on the basis of the relative tolerance parameter $\tau$ as follows:
\begin{equation}
\|w^{(k+1)}-w^{(k)}\| \leq \tau \| w^{(k))} \|
 \label{eq:stop_sb}
\end{equation}
where $w^{(k)} \equiv U^{(j)}$ in the outer loop ($k\equiv j$)  and $w^{(k)} \equiv X^{(m)}$ in the inner loop ($k\equiv m$).
\subsection{Final remarks}
Collecting all the results obtained in the previous paragraphs we  
report the whole scheme of  FCNR in algorithm \ref{alg2} (Table \ref{FNCR}).
In output we have  the computed solution $u^{(\ell)}$, the number $\bar n$ of external iterations 
(the iterations of the continuation procedure with index $\ell$) 
the vector $tot\_it$ and its length $\bar n$. For each index $\ell$,  $tot\_it(\ell)$ reports the number of Forward-Backward iterations performed.

In practice, we obtain a very fast algorithm by performing a few steps of the IR$\ell$1 algorithm for computing  the sequence $\bar u^{(h)}$. Despite the inexact approximation of $u^{(\ell)}$, the final solution is accurate and efficiently computed.

\vspace{5mm}
\begin{table}[!htbp]
\begin{algo}[\textsc{FNCR Algorithm}]
\begin{center}
\begin{tabular}{c}\hline
 \textbf{Input: $r_0,z,\beta$,  Output: $u^{(\ell)},tot\_it,\bar n$
}\quad \quad \quad \quad\quad \quad \quad \quad    \\ 
 \hline 
\begin{minipage}[c]{11cm}
\begin{tabular}{ll}
$ u^{(0)}=\phi^T z, \, w^x=1,\, 
w^y=1, \,\lambda_0=r_0\|u^{(0)}\|_1, \mu=\| \nabla u^{(0)}\|_1$\\[1pt]
$\theta=\frac{0.8}{\beta \| \Delta ^w\|_{\infty}}$ \\ [1pt]
$\tilde u^{(0)}=\hat u^{(0)}=u^{(0)}$;\\[1pt] 
$\ell=0$ \\[1pt]
$\bar n=0$ \\
{\tt repeat}   \ \ \\[2pt] 
\quad  $h=0$, \\[2pt] 
\quad $tot\_it(\ell)=0$ \\[2pt] 
\quad {\tt repeat} \\[2pt]  
\quad  $n=0$\\
\quad {\tt repeat} \\[2pt]
  \quad\quad   $ {\hat v^{(n+1)}}=\hat u^{(n)}+ \beta \Phi^T(z-\Phi \hat u^{(n)}) $ (Forward Step)  \\
\quad \quad [${\tilde u}^{(n+1)}$,it]={\bf fast\_split}($\lambda_{h},\theta,\beta,\hat v^{(n+1)},w^x,w^y$) \\ 
\quad\quad compute $\alpha$ as in \eqref{eq:alpha}  \\
\quad \quad $\hat u^{(n+1)}={\tilde u}^{(n+1)}+\alpha({\tilde u}^{(n+1)}-{\tilde u}^{(n)})$\\[2pt]
 \quad\quad $ n=n+1$ \\[2pt]
 \quad \quad tot\_it($\ell$)=tot\_it($\ell$)+it  (number of FB iterations) \\[2pt]
  \quad {\tt until stopping condition}  as in \eqref{stop} \\[2pt]
\quad $\bar n=\bar n+n$ \\[2pt]
\quad Update $\lambda_{h+1}$ as in \eqref{eq:lambda} \\[2pt]
\quad ${\bar u}^{(h+1)}= {\hat u}^{(n)}$ \\
\quad  Weights updating $w^x,w^y$ as in \eqref{eq:weights}\\[2pt]
\quad ${\hat u}^{(0)}={\hat u}^{(n)}$   \quad (Warm starting) \\[2pt] 
\quad $h=h+1$ \\
\quad {\tt until convergence} \\
$\mu=0.8 \mu $ \\[2pt]
$u^{(\ell+1)}=\bar u^{(h)}$ \\[2pt]
 \quad $\ell=\ell+1$ \\
{\tt until convergence} \\
\end{tabular}
\end{minipage}\\ \hline
\end{tabular}
\end{center}
 \label{alg2}
\end{algo} 
\caption{FNCR Algorithm.}
\end{table}
 \section{Numerical Experiments \label{numres}}
In this section we report the results of several tests run on simulated under-sampled 
data obtained by synthetic (phantoms) and  full resolution MRI images.
The   under-sampled data $z$ are obtained as $z=\Phi u$ where $u$ is the full resolution  image  and
$\Phi$ is the under-sampling Fourier matrix,  obtained as in \eqref{eq:mask}.
The  under-sampling masks,  analyzed in the  next paragraphs are: radial mask ($\mathcal{M}_1$), parallel mask ($\mathcal{M}_2$) and random mask ($\mathcal{M}_3$).
In figure \ref{fig:uno} we represent  an example of each mask with low sampling rate $S_r$, 
measured by the percentage ratio between the number of non-zero pixels $N_{\mathcal{M}}$ and the total number of pixels $N^2$:
\begin{equation}
S_r=\frac{N_{\mathcal{M}}}{N^2} \%.
\label{eq:Sr}
\end{equation}
The quality of the reconstructed  image $u$ is evaluated by means of the Peak Signal to Noise Ratio (PSNR)
$$ PSNR=20 \log_{10}   \frac{\max(x)}{rmse}, \ \ \ \ rmse=\sqrt{\frac{\sum_i \sum_j ( u_{i,j}-x_{i,j})^2}{N^2}} $$
where $rmse$ represents the root mean squared error and $x$ is the true image.
All the algorithms are implemented in Matlab R16 on a PC equipped with Intel 7 processor and 8GB Ram.

In the first two paragraphs (\ref{alg_perf}, \ref{noise_par}) we test FNCR on synthetic data in order to asses
its performance in case of critical under-sampling both on noiseless and noisy data.
Finally in paragraph \ref{comp} data from real MRI images are used to compare FNCR to one of the most efficient 
reconstruction algorithms \cite{IL1MRI}.
\subsection{Algorithm performance \label{alg_perf} on noiseless data}
In this paragraph we test the performance of FNCR algorithm in reconstructing good quality images from 
highly under-sampled data. We  focus on two synthetic images: the \texttt{Shepp-Logan} phantom (T1) (figure \ref{fig:f1}(a)), widely used 
in algorithm testing, and the \texttt{Forbild} phantom (T2) \cite{Yu2012} (figure \ref{fig:f1}(b)), well known as a very difficult test problem.

In this first set of tests we  stop  the outer iterations of the FNCR algorithm \ref{alg2}
 as soon as $PSNR \geq 100$ or  $\bar n > 5000$. 
The value $PSNR=100$ indicates an extremely good quality reconstruction since our test
images are scaled in the interval $[0,1]$, therefore  $rmse \simeq 10^{-5}$.

In table \ref{tab:tabPhi_all} we report the results obtained for the different test images and masks.
In column $1$ the type of mask is reported, column $2$ contains the sampling
rate $S_r$ of each test (in brackets the number of rays or lines for masks 
$\mathcal{M}_1$ or $\mathcal{M}_2$, respectively), finally for each test image,
T1 (columns $3-4$)  and T2 (columns $5-6$) we report  the PSNR of the starting image $\Phi^t z$ ($PSNR_0$) and 
the value $\bar n$, computed as in Algorithm \ref{alg2} and measuring the computational cost.
When the algorithm stops because the maximum number of iterations has been reached ($\bar n=5000$), we report between brackets the last PSNR obtained.
\begin{table}[!htbp]
\centering
\begin{tabular}{c c c l c l}
\hline
\multicolumn{1}{l}{} & \multicolumn{1}{l}{} & \multicolumn{2}{c}{T1} & \multicolumn{2}{c}{T2} \\
 & $ S_r $ & $PSNR_0$ &  $\bar n(PSNR)$ &  $PSNR_0$ &  $ \bar n (PSNR)$  \\
\hline
 $\mathcal{M}_1$   & $3\%(7)$      & 15.7 & 4500 & 15.02  & 5001 (28.4)\\
            & $5.7\%(12)$                  & 16.8 & 339 & 16.93  & 378\\
            & $10.7\%(18)$                   & 18.8 & 190 & 18.69  & 233\\
            & $25 \%(60)$                    & 22.6 &  96 &  21.8  & 131\\
\hline
 $\mathcal{M}_2$   & $6.3\% (16)$  & 14.6 & 1309 &12.8  & 5001(29.95)\\
            & $12.5\%(32) $      & 16.2 & 778 &  13.9 & 386 (30.9)\\
            & $25 \%(64)$ &  & 209   &  19.1 & 121 18.5\\
\hline
 $\mathcal{M}_3$   &  $2 \%$   & 16.2 & 419 & 15.98  & 5001 (31.7)\\
            & $12\%$     & 17.1 & 106 &  16.7 & 128\\
            & $25 \%$  & 18.3 & 82  &  18.6 & 114\\
\hline
\end{tabular}
\caption{FNCR results with masks $\mathcal{M}_1-\mathcal{M}_3$ on noiseless data. 
$S_r$ is the Sampling Ratio as in \eqref{eq:Sr},
$PSNR_0$ refers to the PSNR of the starting image $\Phi^t z$,
$\bar n$ is the total number of Forward-Backward iterations in Algorithm \ref{alg2}. In brackets the number of rays or lines for masks 
$\mathcal{M}_1$ or $\mathcal{M}_2$, respectively. }
\label{tab:tabPhi_all}
\end{table}
We observe that T1 always reaches a PSNR value greater than  100  even with very severe under-samplings. 
The T2  confirms to be a very difficult test problem especially in case of $\mathcal{M}_2$ under-sampling mask.

We conclude our analysis of the FNCR performance with a few observations about the input parameters 
of Algorithm \ref{alg2}.
In table \ref{tab:par0} we report  the values of $r_0$, $\gamma$ and $\beta$ used in the present reconstructions and
 notice that  the parameters $r_0$ and $\gamma$ depend on the mask.
We observe that the algorithm uses the same parameters for  $\mathcal{M}_2$ and $\mathcal{M}_3$, 
  while it requires smaller values of  $r_0$ and $\gamma$ for 
 $\mathcal{M}_1$.
 The parameter $\beta$ is mainly affected by the data noise and in the present experiments $\beta=1$ gives always 
 the best results.
\begin{table}[h]
\centering
\begin{tabular}{c c c c c c}
\hline
 & $r_0$ & $\gamma$ & $\beta$ \\
\hline
 $\mathcal{M}_1$  & $10^{-4}$ & $5 \cdot 10^{-2}$ & 1 \\
 $\mathcal{M}_2$,$\mathcal{M}_3$   & $5 \cdot 10^{-2}$ & $5 \cdot 10^{-1}$ & 1 \\
\hline
\end{tabular}
\caption{FCNR input parameters for noiseless data.}
\label{tab:par0}
\end{table}

Concerning the value of the tolerance parameter $\tau$ in the stopping criterium \eqref{eq:stop_sb} we  analyze here how it affects the efficiency 
 of the proposed  matrix splitting method \eqref{eq:it_new} both
in terms of accuracy and computational cost. 

As  representative examples we report  the results obtained with $\tau \in [10^{-4}, 10^{-1}]$ 
in the case of radial sampling $\mathcal{M}_1$ with $S_r=3\%$ for both T1 and T2.
As shown in figure \ref{fig:PSNR_tol_SL}, where different lines represent the $PSNR$ obtained as a function of the iterations with different values of $\tau$,
we have two  possible situations: in the 
T1 test  (figure \ref{fig:PSNR_tol_SL}(a)) the  PSNR value 100  is always obtained, while in 
the second case (figure \ref{fig:PSNR_tol_SL}(b)) the  $PSNR$ never gets the expected value  100.
Regarding the computational cost, 
we plot in figure \ref{fig:ITER_tol} the number $tot\_in$ of the total iterations required in step $\ell$ (as computed in Algorithm \ref{alg2}) as a function of the iterations, using different lines for different values of $\tau$.  
We note that the computational cost is minimum when $\tau \geq 10^{-3}$ and greatly increases when  $\tau\leq 10^{-4}$, hence a small tolerance parameter $\tau$ is never convenient since it increases the computational cost
and it does not help to improve  the final PSNR.
We can conclude that  very few iterations of the iterative method are sufficient to solve the linear system \eqref{Slin1} with good accuracy.

Therefore, all the  experiments reported in the present section are obtained using  $\tau = 0.1$. 
In this case only $4$ inner iterations $\bar m$ are required by  Algorithm \ref{alg3} for each FB step and
the total computational cost of FNCR is  $\bar n \cdot 4$. 

\begin{figure}
\begin{center}
\subfigure[$\mathcal{M}_1, S_r=3\%$]{
 \includegraphics[width=5cm,height=5cm]{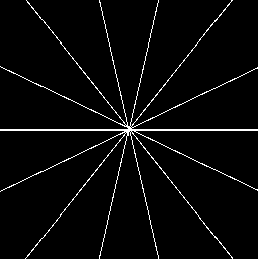}}
 \subfigure[$\mathcal{M}_2, S_r=6\%$]{
  \includegraphics[width=5cm,height=5cm]{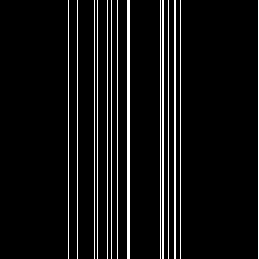}}
  \subfigure[$\mathcal{M}_3, S_r=2\%$]{
  \includegraphics[width=5cm,height=5cm]{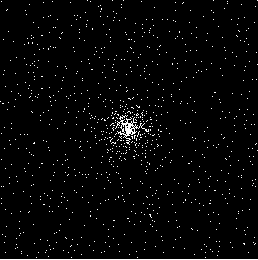}}
\end{center}
\caption{Sampling Masks}
\label{fig:uno}
\end{figure}
\begin{figure}
 \centering
 \includegraphics[width=6cm,height=6cm]{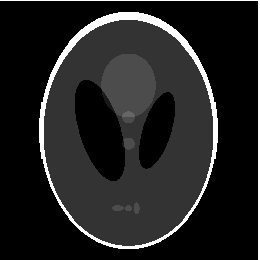}
 \includegraphics[width=6cm,height=6cm]{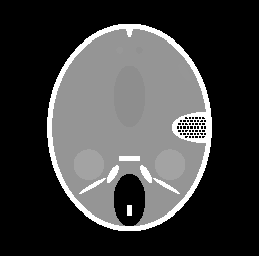}\\
 (a) \hspace{5cm}    (b)\\
 \caption{Test Images. (a) Shepp Logan (T1), (b) Forbild (T2) }
 \label{fig:f1}
\end{figure}
\begin{figure}[h]
 \centering
 \hspace{-10mm}  \subfigure[T1 test]{
 \includegraphics[width=6cm,height=5cm]{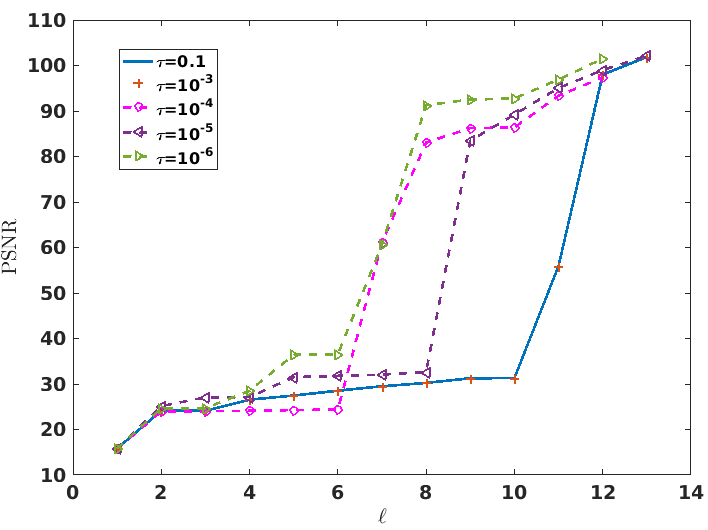}}
   \subfigure[T2 test]{
 \includegraphics[width=6cm,height=5cm]{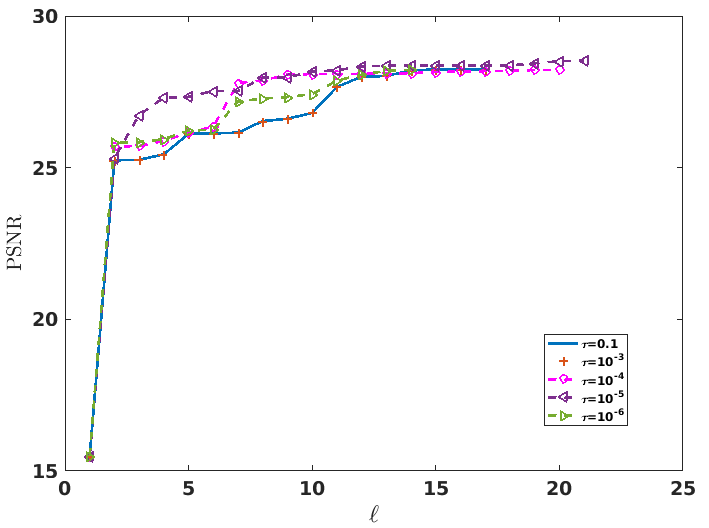}}
 \caption{PSNR vs. the iterations $\ell$,  for different tolerance parameters $\tau$ in \eqref{eq:stop_sb}. Mask $\mathcal{M}_1$ with $S_r=3\%$. 
 Tolerances $\tau=0.1$ and $\tau=10^{-3}$ are represented 
by blue line and red crosses;   $\tau=10^{-4}$ is represented by magenta line and circles; $\tau=10^{-5}$ is represented by 
cyan line and left triangles finally $\tau=10^{-6}$ is represented by 
green line and right triangles.
  }
\label{fig:PSNR_tol_SL}
 \end{figure}
 
\begin{figure}[h]
 \centering
 \hspace{-10mm}   \subfigure[T1 test ]{
 \includegraphics[width=6cm,height=5cm]{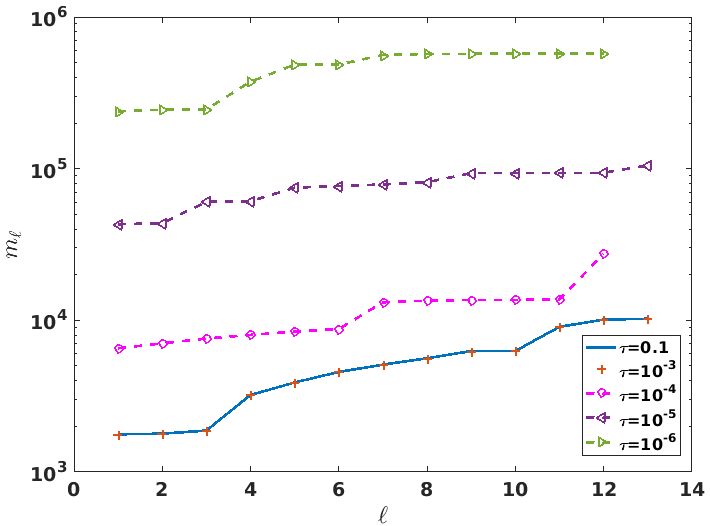}}
  \subfigure[T2 test]{
 \includegraphics[width=6cm,height=5cm]{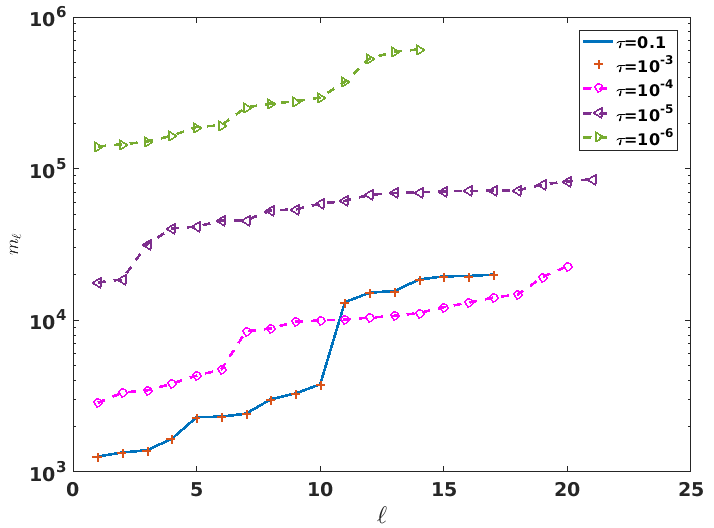}}
 \caption{ $Tot\_in$ vs. the iterations $\ell$ for different tolerance parameters $\tau$ in \eqref{eq:stop_sb}. 
Mask $\mathcal{M}_1$ with $S_r=2\%$. Left: Shepp-Logan test. Tolerances $\tau=0.1$ and $\tau=10^{-3}$ are represented 
by blue line and red crosses;   $\tau=10^{-4}$ is represented by magenta line and circles; $\tau=10^{-5}$ is represented by 
cyan line and left triangles finally $\tau=10^{-6}$ is represented by 
green line and right triangles.
}
\label{fig:ITER_tol}
 \end{figure}
\subsection{Noise Sensitivity \label{noise_par}}
In this paragraph, we explore the noise sensitivity of the proposed FNCR method by reconstructing the test images from noisy data
with different under-sampling masks.
The simulated noisy data $z^{\delta}$ are obtained by adding  to the undersampled  data $z$ normally distributed noise 
of  two different levels $\delta = 10^{-3}, 10^{-2}$:
\begin{equation}
z^{\delta} = z + \delta \| z \| v, \ \ \ \| v \| =1
\label{eq:noise}
\end{equation}
where  $v$ is a unit norm random vector obtained by the Matlab function {\tt randn}.
Regarding the FNCR parameters, we used the same values as in table \ref{tab:par0}
for $\mathcal{M}_2$ and $\mathcal{M}_3$ masks, whereas  we used $r_0=10^{-2}$ and $\gamma=0.2$ for mask $\mathcal{M}_1$.

In table  \ref{tab:tabPhi_all_noise} we report the best results obtained by the FNCR algorithm in terms of $PSNR$ and number $\bar n$ of FB iterations. 

Compared to the results obtained in the noiseless case (table \ref{tab:tabPhi_all}), we observe that  the 
$PSNR$ value is smaller but still acceptable ($>30$) in most cases for the noise level $\delta= 10^{-3}$. 
For the higher noise level  $\delta= 10^{-2}$, the reconstructions are more difficult, especially for low sampling rates.

Concerning the radial sampling $\mathcal{M}_1$ and random sampling $\mathcal{M}_3$, the Forbild data (T2) are badly reconstructed for $S_r=4.3\%$
and $S_r = 3\%$ respectively, 
where the Shepp-Logan data  (T1) always reach $PSNR > 30$. 
The most critical problem is the parallel under-sampling $\mathcal{M}_2$ with $\delta =10^{-2}$, in this case 
both T1 and T2 reach $PSNR > 30$ only  with $S_r \geq 25\%$. 
Moreover  $\mathcal{M}_2$ with $\delta =10^{-3}$ is critical only for Forbild data (T2) with $S_r = 6.3\%$. 
Finally,  for each test (T1, T2) we report in figures  \ref{fig:T1_noise} and \ref{fig:T2_noise}
the images whose $PSNR$ values are  
typed in bold  in table \ref{tab:tabPhi_all_noise}. 
We can appreciate the overall good  quality of the reconstructions from noisy data.
\begin{table}[!htbp]
\centering
\begin{tabular}{c c c c c c}
\hline
\multicolumn{1}{l}{} & \multicolumn{1}{l}{} &  \multicolumn{1}{l}{} & \multicolumn{1}{c}{T1} & \multicolumn{1}{c}{T2} \\
 & $ \delta $ & $S_r$ &  $PSNR(\bar n)$ &  $PSNR(\bar n)$  \\
\hline
 $\mathcal{M}_1$   & $10^{-3}$ & $4.3\%$    &   59.61 (130) & 30.85 (192)\\
            &           & $7.7\%$   &   66.23 (72) & 60.78 (66 ) \\ 
            &           & $10.7\%$    &   68.61 (57) & 62.84(807)  \\
            & $10^{-2}$ & $4.3\%$     &   39.5 (461) & 24.56 (44) \\
            &           & $7.7\%$    &   43.69 (35) & 36.82 (117) \\
            &           & $10.7\%$   &   46.36 (27) & {\bf 41.06} (405)\\ 
\hline
 $\mathcal{M}_2$   & $10^{-3}$ & $6.3\%$               &   34.08 (1210) &  20.2 (768)\\
            &           & $12.5\%$              &    55.95 (590) &  30.81 (228) \\ 
            &           & $25\%$                &    69.75 (68)  &  34.76 (67) \\
            & $10^{-2}$ & $6.3\%$               &    22.5 (962) &  17.93 (578)\\
            &           & $12.5\%$              &   26.87 (60) &  {\bf 22.54} (63) \\
            &           & $25\%$                &  {\bf 52.7} (52)&  33.04 (44)\\ 
\hline
 $\mathcal{M}_3$   & $10^{-3}$ & $3\%$      &   64.4 (462)   & 53.5(485)\\ 
            &           & $12\%$     &   64.81 (242) & 62.45 (63) \\ 
            &           & $25\%$   &  71.45 (58) & 65.27 (241)   \\
            & $10^{-2}$ & $3\%$     &  33.54 (134) & 23.55 (418)\\ 
            &           & $12\%$     &     {\bf 43.29} (37) & 42.99 (45) \\
            &           & $25\%$    &   50.57 (42)& 45.72 (225)\\ 
\hline            
\end{tabular}
\caption{FNCR results with masks $\mathcal{M}_1-\mathcal{M}_3$ on noisy data. $\delta$ is the noise level \eqref{eq:noise}, 
$S_r$ is the Sampling Ratio \eqref{eq:Sr} and 
$\bar n$ is the total number of Forward-Backward iterations in Algorithm \ref{alg2}.}
\label{tab:tabPhi_all_noise}
\end{table}
\begin{figure}
 \centering
 \includegraphics[width=6cm,height=6cm]{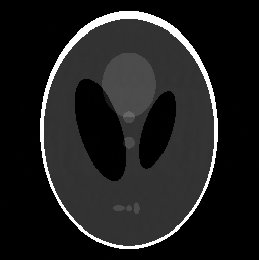}
 \includegraphics[width=6cm,height=6cm]{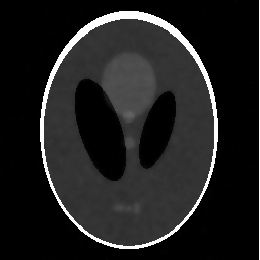}\\
 (a) \hspace{5cm}    (b)\\
 \caption{Images reconstructed from T1 noisy data. (a) $PSNR=52.7$ (b) $PSNR=43.29$. }
 \label{fig:T1_noise}
\end{figure}
\begin{figure}
 \centering
 \includegraphics[width=6cm,height=6cm]{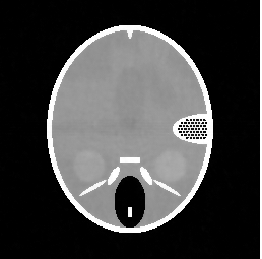}
 \includegraphics[width=6cm,height=6cm]{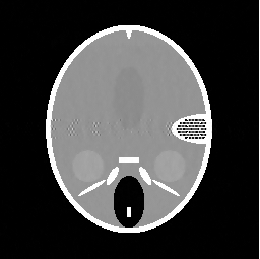}\\
 (a) \hspace{5cm}    (b)\\
 \caption{Images reconstructed from T2 noisy data. (a) $PSNR=41.06$  (b)$PSNR=22.54$. }
 \label{fig:T2_noise}
\end{figure}

\subsection{Comparison with other methods \label{comp}}
In this paragraph we compare the FNCR algorithm with the very efficient algorithm  IL$_1-\ell_q$
recently proposed in the literature \cite{IL1MRI}, 
where the authors minimize a class
of concave sparse metrics $\ell_q (q=0.5)$ in a general DCA-based CS framework.
Both synthetic and real MRI data are tested.
In \cite{IL1MRI} the authors test their method on T1 image with radial mask $\mathcal{M}_1$ and they show that a perfect ($PSNR\geq 100$) can be obtained with 7 rays. 
Performing the same test  FNCR also reaches  $PSNR \geq 100$ in comparable times.

Then we applied both methods to the
 T2 test image both methods with mask $\mathcal{M}_1$ and we obtained perfect reconstructions with $10$ sampled projections, while
decreasing the projections to $8$  IL$_1-\ell_q$ obtains $PSNR=28.7$ and FNCR reaches $PSNR=98.1$
(see Figure \ref{fig:T2_8r}). 
\begin{figure}
 \centering
 \includegraphics[width=6cm,height=6cm]{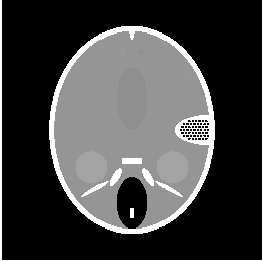} 
  \includegraphics[width=6 cm,height=6 cm]{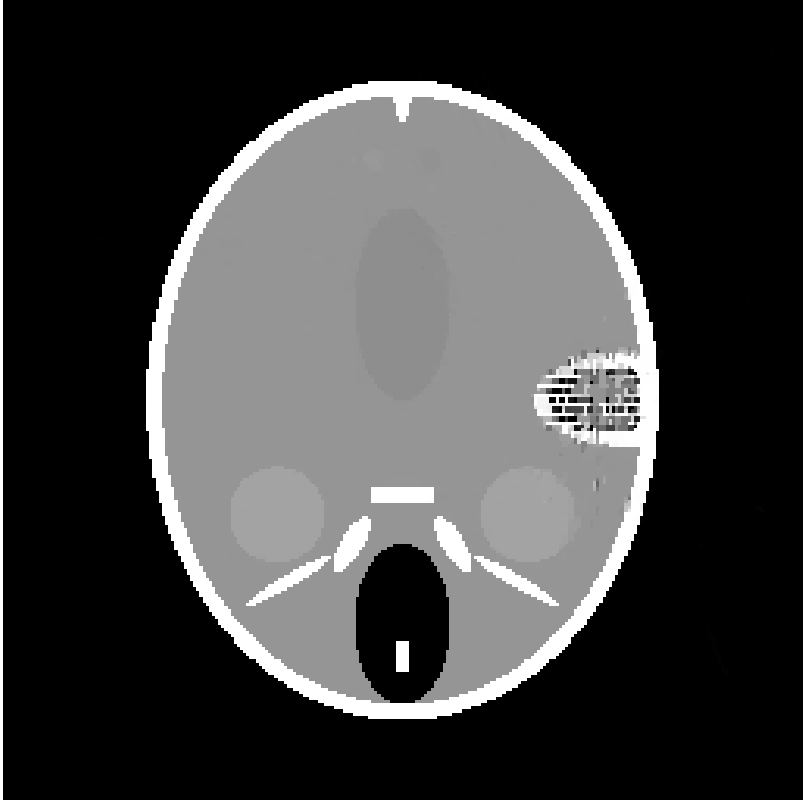}\\
 (a) \hspace{5cm}    (b)\\
 \caption{T2 test image, radial sampling mask ($\mathcal{M}_1$) $8$ lines ($S_r=3.98$): (a) FNCR reconstruction, (b) IL$_1-\ell_q$ reconstruction.}
 \label{fig:T2_8r}
\end{figure}

\begin{figure}
 \centering
 \includegraphics[width=6cm,height=6cm]{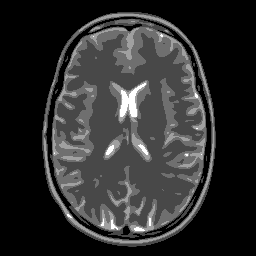}
 \caption{MRI $256 \times 256$ brain image.}
 \label{fig:brain}
\end{figure}
Concerning the real MRI data we compare  IL$_1-\ell_q$ and FNCR algorithms in the  reconstructions of  the $256 \times 256$
brain image (T3 test), represented  in figure \ref{fig:brain}. We report in table \ref{tab:tabComp_B} the results 
obtained by reconstructing the noiseless data undersampled by $\mathcal{M}_1$, $\mathcal{M}_2$, $\mathcal{M}_3$ masks.

From the table, we see that FNCR always outperforms  IL$_1-\ell_q$. 

In figure \ref{fig:B3_8p} we show the FCNR and  IL$_1-\ell_q$ reconstructions 
in case 
of mask $\mathcal{M}_3$ and  $Sr=8\%$. 

Finally in figure \ref{fig:Row_err}  
we plot the reconstructed rows corresponding to the largest error (200-th row) for the true image and
the FNCR reconstruction (Figure \ref{fig:Row_err}(a)) and for the true image and
 IL$_1-\ell_q$ reconstruction (Figure \ref{fig:Row_err}(b)). 
 We observe that FNCR better fits the corresponding row of the original  image. 
\\

\begin{table}[h]
\centering
\begin{tabular}{c c c c c c}
\hline
\multicolumn{1}{l}{} &\multicolumn{1}{l}{} & \multicolumn{2}{c}{IL$_1-\ell_q$} & \multicolumn{2}{c}{FNCR} \\
& $S_r$ & PSNR & Time & PSNR & Time \\
\hline 
$ \mathcal{M}_1$& $10.60\% $ & 100.29 & 13.46s  &  100.3 & 13.88  \\
&  $9.02\%$ & 27.68 & 28.70 & 32.24 & 27.67s   \\
\hline            

\hline 
$ \mathcal{M}_2$& $25\%$ & 50.41 & 6.15s  & 60.54 & 47.98s \\    
&  $12.5\%$  & 33.87 & 19.32s & 34.74 & 87.84s\\    
\hline      

$ \mathcal{M}_3$ & $10\%$ &100.01 & 4.22s & 100.1 & 6.26s \\   
& $8\%$  & 30.22 & 28.45s &100.2 & 37.34s\\
& $5\%$ & 25.46 & 28.06 & 27.46 & 6.77s \\
\hline                  
\end{tabular}
\caption{Comparison between FNCR and  IL$_1-\ell_q$ reconstructions on T3 test using $\mathcal{M}_1$, $\mathcal{M}_2$, $\mathcal{M}_3$ masks.}
\label{tab:tabComp_B}
\end{table}

\begin{figure}
\begin{center}
 \includegraphics[width=6cm,height=6cm]{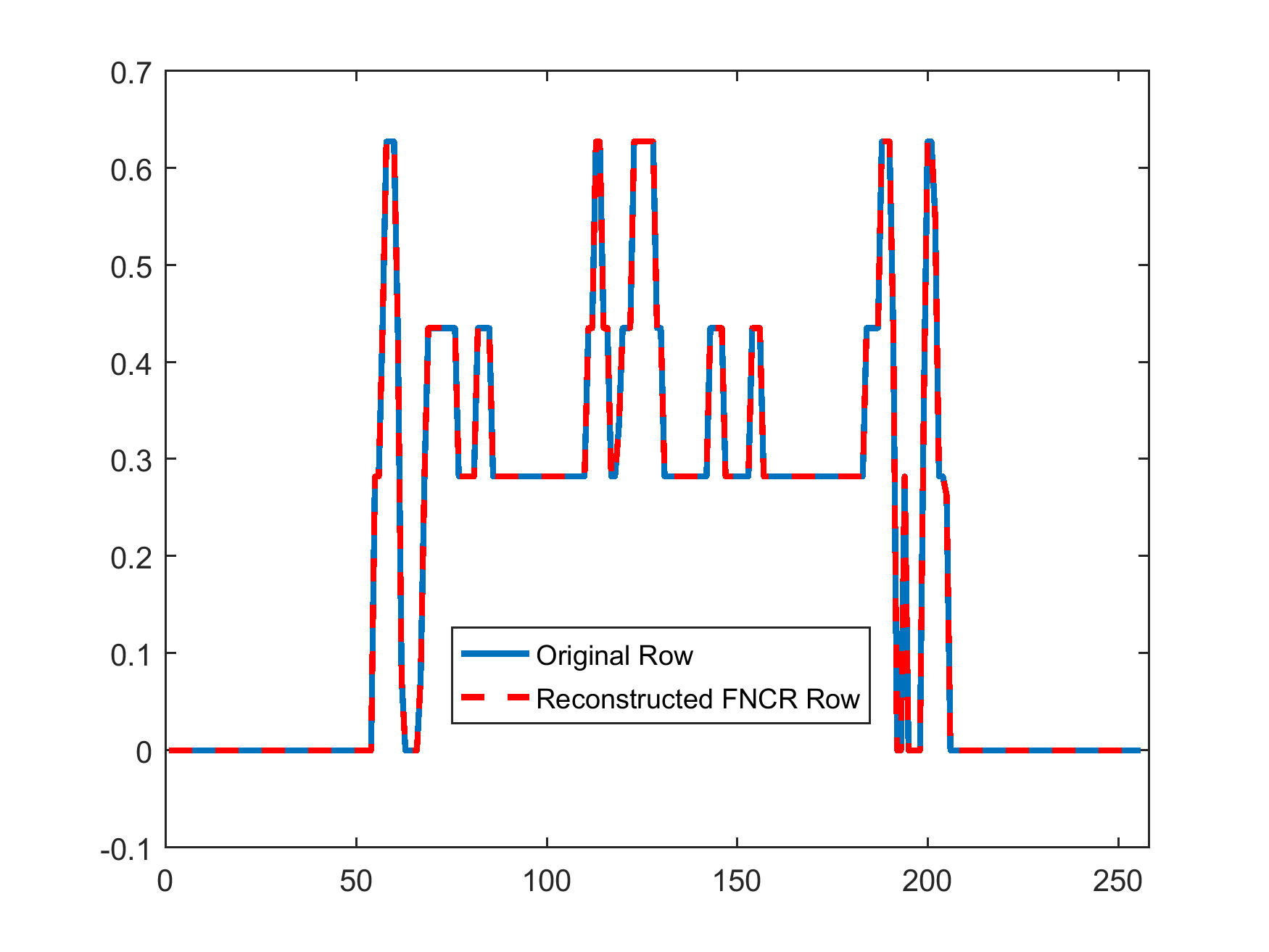}
 \includegraphics[width=6cm,height=6cm]{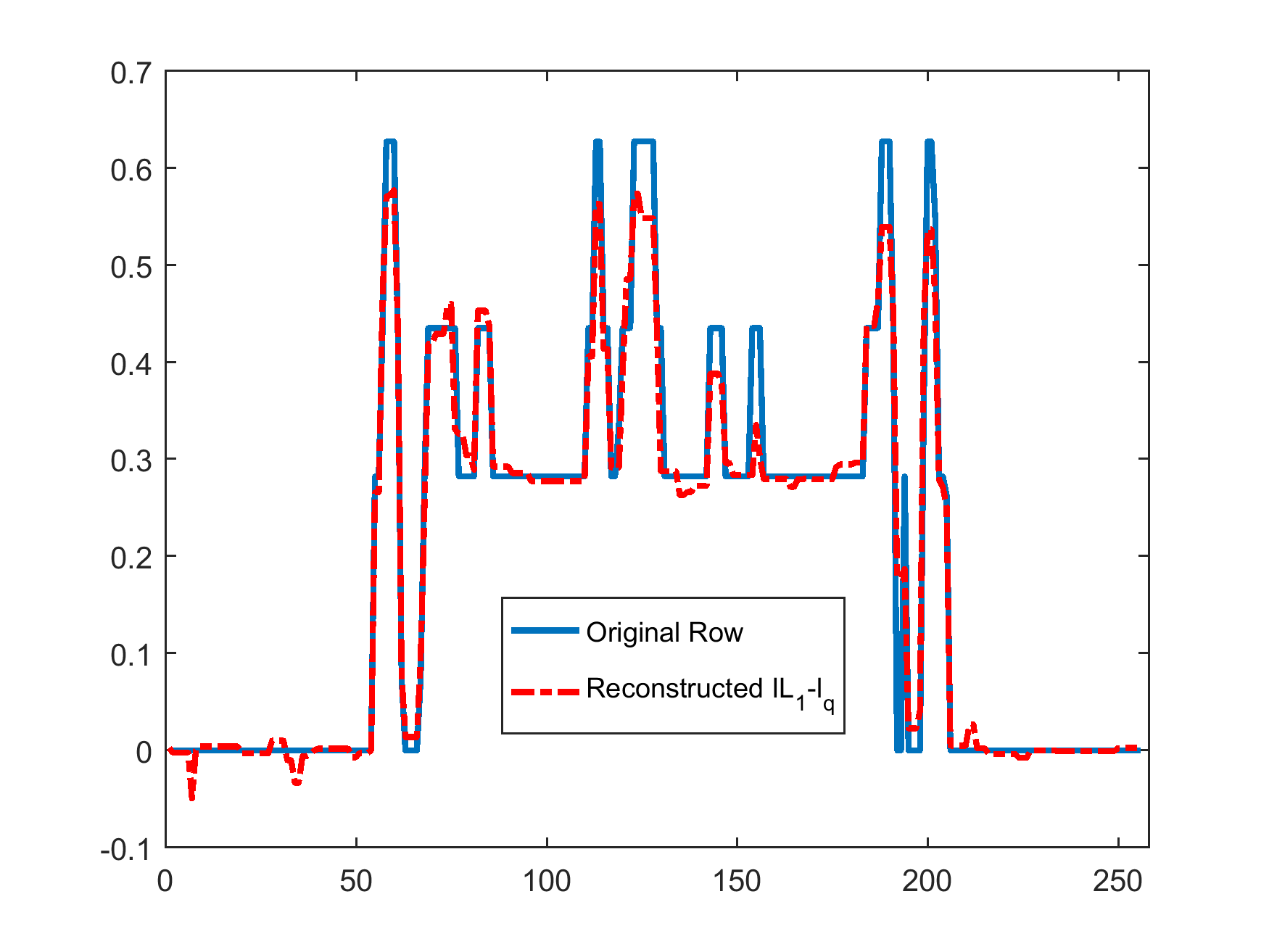}\\
(a) \hspace{5cm}    (b)\\

\end{center}
\caption{Row 200 of T3 image. Reconstructions obtained with $\mathcal{M}_2$ $Sr=8\%$and Mask undersampling $\mathcal{M}_3$, $Sr=8\%$. 
(a) original and  FNCR reconstruction, (b) original and  IL$_1-\ell_q$ reconstruction.}
\label{fig:Row_err}
\end{figure}

\begin{figure}
 \centering
 \includegraphics[width=6cm,height=6cm]{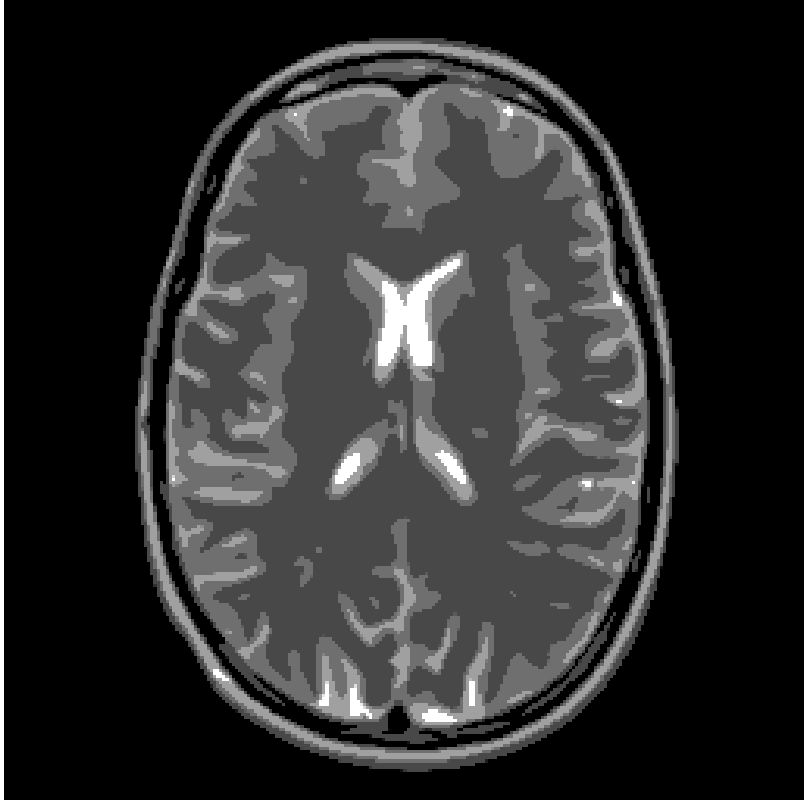} 
 \includegraphics[width=6cm,height=6cm]{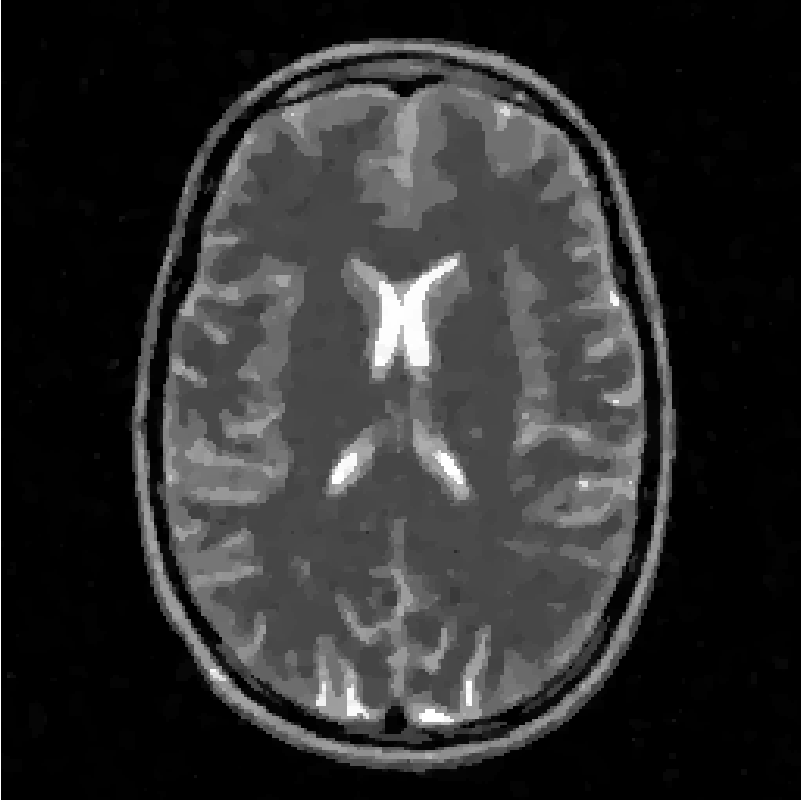} \\
 (a) \hspace{5cm}    (b)\\
 \caption{T3 image reconstructions with $\mathcal{M}_3$ and $S_r=8\%$: (a) FNCR reconstruction, (b) IL$_1-\ell_q$ reconstruction.}
 \label{fig:B3_8p}
\end{figure}
We can finally conclude  that FNCR is more competitive than IL$_1-\ell_q$ in reconstructing difficult test images from very severely undersampled data.
\section{Conclusions \label{concl}}
We presented
 a new algorithm,  called Fast Non Convex Reweighting,
 to solve a non-convex minimization problem by means of  iterative convex approximations.
 The resulting convex  problem is solved by an iterative Forward-Backward (FB) algorithm, 
 and a Weighted Split-Bregman iteration is applied in the Backward step.   
 A new splitting scheme is proposed to improve the efficiency of the Backward steps. 

 A detailed analysis of the algorithm properties and  convergence is carried out.
 The extensive experiments on image reconstruction from undersampled  MRI data show that 
 FNCR gives very good results in terms both of quality and computational complexity and it   
 is more competitive in reconstructing  test problems with real images from very severely under-sampled data.
 \appendix
\section{Convergence \label{A}}
Aim of this subsection is to prove that the iterates $\bar{u}^{(h)}$, computed by Algorithm \ref{alg1},
converge to a local minimum of problem   \eqref{eq:nconv}.
\\
To this purpose we consider both the nonconvex functional 
\begin{equation}
{\cal P}(\lambda_h,F_{\mu_{\ell}} (Du),u)=  \left \{ \lambda_h F_{\mu_{\ell}}(Du)  +\frac{1}{2}||\Phi u- z||_2^2 \right \}
\label{nconv}
\end{equation}
and its convex majorizer
\begin{equation}
{\cal P}(\lambda_h,{\cal F}_{h,\mu_{\ell}}(Du),u)=  \left \{ \lambda_h {\cal F}_{h,\mu_{\ell}}(Du)  +\frac{1}{2}||\Phi u- z||_2^2 \right \}.
\label{convMaj}
\end{equation}
\par\noindent
The approximate solution, computed by  Algorithm \ref{alg1} in \eqref{eq:alg1},
can be written as:
\begin{equation}
\bar{u}^{(h)}= \arg \min_{u} {\cal P}(\lambda_{h-1},{\cal F}_{h-1,\mu_\ell}(Du),u).
\label{umin}
\end{equation}  
Since ${\cal F}_{h-1, \mu_\ell}(Du)$ is a local convex approximation of the nonconvex function $F_{\mu_\ell}(Du)$ 
at each step $h$ of the iterative reweighting method, then  from \eqref{eq:fbig} we have:
\begin{equation}
 {\cal F}_{h-1,\mu_\ell}(D\bar{u}^{(h-1)})=F_{\mu_\ell}(D\bar{u}^{(h-1)})  \quad \mbox{ and} \quad \quad  
 {\cal F}_{h-1, \mu_\ell}(Du) \ge F_{\mu_\ell}(Du) \ \ \ u \neq \bar{u}^{(h-1)}.
\label{Maj}
\end{equation}  
The following result proves the descent property of the iterative reweighting method. 
\begin{prop}
Let $\lambda_{h-1}$ and $\lambda_h$ be two values of the penalization parameter such that $\lambda_{h-1} > \lambda_h$
and let $\bar{u}^{(h)}$ and $\bar{u}^{(h+1)}$ be the corresponding minimizers given by  ({\ref{umin}}), 
then the nonconvex functional ${\cal P}(\lambda,F_{\mu_\ell} (Du),u)$ satisfies the following inequality:
$${\cal P}(\lambda_{h-1},F_{\mu_\ell} (D\bar{u}^{(h)}),\bar{u}^{(h)})>
{\cal P}(\lambda_{h},F_{\mu_\ell} (D\bar{u}^{(h+1)}),\bar{u}^{(h+1)})$$
\label{propos4-1}
\end{prop}
{\bf Proof:} 
Using relation \eqref{nconv}  and the first part of property \eqref{Maj}, rewritten for $h$ we have:
\begin{align*}
 {\cal P}(\lambda_{h-1},F_{\mu_\ell}(D\bar{u}^{(h)}),\bar{u}^{(h)}) & = \frac{1}{2} ||\Phi \bar{u}^{(h)}-z||_2^2+\lambda_{h-1} 
F_{\mu_\ell}(D\bar{u}^{(h)}) \\
 & = \frac{1}{2} ||\Phi \bar{u}^{(h)}-z||_2^2 +\lambda_{h-1} {\cal F}_{h,{\mu_\ell}}(D\bar{u}^{(h)})
\end{align*}
From the assumption  $\lambda_{h-1} > \lambda_h$ it follows that:
$$\frac{1}{2} ||\Phi \bar{u}^{(h)}-z||_2^2 +\lambda_{h-1} {\cal F}_{h,\mu_\ell}(D\bar{u}^{(h)}) > 
\frac{1}{2} ||\Phi \bar{u}^{(h)}-z||_2^2 +\lambda_{h} {\cal F}_{h,\mu_\ell}(D\bar{u}^{(h)})$$
Using \eqref{umin}:
$$\frac{1}{2} ||\Phi \bar{u}^{(h)}-z||_2^2 +\lambda_{h} {\cal F}_{h,\mu_\ell}(D\bar{u}^{(h)}) \ge
\frac{1}{2} ||\Phi \bar{u}^{(h+1)}-z||_2^2 +\lambda_{h} {\cal F}_{h,\mu_\ell}(D\bar{u}^{(h+1)}) $$
Applying the second part of  property \eqref{Maj} to the case $Du\equiv D\bar{u}^{(h+1)}$:
$$ \frac{1}{2} ||\Phi \bar{u}^{(h+1)}-z||_2^2 +\lambda_{h} {\cal F}_{h,\mu_\ell }(D\bar{u}^{(h+1)}) \ge
\frac{1}{2} ||\Phi \bar{u}^{(h+1)}-z||_2^2 +\lambda_{\ell} F_{\mu_\ell}(D\bar{u}^{(h+1)})$$
From definition \eqref{nconv}: 
$$ \frac{1}{2} ||\Phi \bar{u}^{(h+1)}-z||_2^2 +\lambda_{h} F_{\mu_{\ell}}(D\bar{u}^{(h+1)})={\cal P}(\lambda_{h},F_{\mu_\ell}(D\bar{u}^{(h+1)}),\bar{u}^{(h+1)})$$
hence:
$${\cal P}(\lambda_{h-1},F_{\mu_\ell}(D\bar{u}^{(h)}),\bar{u}^{(h)}) > {\cal P}(\lambda_{h},F_{\mu_\ell}(D\bar{u}^{(h+1)}),\bar{u}^{(h+1)})$$
This proves the result.
\qed

In order to prove the convergence we need the following results.
\begin{prop}
Let's define  the bounding set $B$ s.t. $\bar{u}^{(h)} \in B$, $\forall h \ge 0$, 
let's assume that $F_{\mu_\ell}$ and ${\cal F}_{h,\mu_\ell}$ have a locally Lipschitz continuous gradient on $B$
 with a common Lipschitz constant $\tilde{L}\geq 0$, 
 let's assume also that ${\cal P}(\lambda_h,{\cal F}_{h,\mu_\ell}(Du),u)$ is strongly convex with convexity parameter $\nu>0$, 
 then the following properties hold:
\begin{enumerate}
\item  Descent property:
\begin{equation}
\label{dscp}
{\cal P}(\lambda_{h-1},F_{\mu_{\ell}} (D\bar{u}^{(h)}),\bar{u}^{(h)})>{\cal P}(\lambda_{h},F_{\mu_{\ell}}D\bar{u}^{(h+1)}),\bar{u}^{(h+1)})
\end{equation}
\item  There exists $C>0$ such that for all $\ell \in N$ there exists 
$$\xi_{h+1} \in \partial  {\cal P}(\lambda_{h},F_{\mu_\ell} (D\bar{u}^{(h+1)}),\bar{u}^{(h+1)})$$ fulfilling 
\begin{equation}
\label{zerop}
\| \xi_{h+1}\|_2 \leq C \| \bar{u}^{(h+1)}-\bar{u}^{(h)}\|_2
\end{equation}
\item For any converging subsequence $\left \{\bar{u}^{(h_j)}\right \}_{j\in N}$ with 
$$\bar{\bar u}\equiv \lim_{j \rightarrow \infty}\bar{u}^{(h_j)}, $$ and $\left\{\lambda_{h_j}\right\}$ with
$\bar{\bar \lambda}=\lim_{j \rightarrow \infty}\lambda_{h_j}$, 
we have 
\begin{equation}
\label{cvp}
{\cal P}(\lambda_{h_j},F_{\mu_{\ell}} (D\bar{u}^{(h_j)}),\bar{u}^{(h_j)}) \rightarrow {\cal P}(\bar{\bar \lambda},F_{\mu_\ell} 
(D \bar{\bar u}), \bar{\bar u})
 \quad j \rightarrow \infty.
\end{equation}
\end{enumerate}

\end{prop}
\par\noindent
{\bf Proof:}
Using proposition \ref{propos4-1} we can prove the descent property \eqref{dscp}. 
Relations \eqref{zerop} and \eqref{cvp} can be easily proved as in \cite{Pock2015}, proposition 5.
\qed
\begin{prop}
 The sequence $\bar{u}^{(h)}$ generated by Algorithm \eqref{alg1}  converges to a local minimum of problem  \eqref{eq:nconv} for each $\ell$.
\end{prop}
{\bf Proof:}
We observe that the objective function  
$$\left \{ \lambda_h F_{\mu_\ell}(Du) + \frac{1}{2} \| \Phi u - z \|_2^2 \right \},\ \  \lambda >0.$$
satisfies the Kurdika-Lojasiewicz property
\cite{Attouch13}.
In fact it is an analytic function, because the function $\psi$ in \eqref{fell} is evaluated in nonnegative arguments, therefore
we can restrict the domain of $\psi$ on $R^+ \cup \{0 \}$.
\\
Finally the convergence  immediately follows from \cite{Pock2015}, where the authors prove that the 
Iterative Reweighted Algorithm converges, 
provided that the relations \eqref{dscp}, \eqref{zerop}, \eqref{cvp} hold
and the objective function satisfies the Kurdika-Lojasiewicz
property.
\qed
\section{Proof of theorem \ref{th:t1} \label{A_B}}
In order to prove theorem \ref{th:t1} we first prove the following lemma. 
\begin{lemm}
\label{lm:l1}
Let $M=E+F$ where $E$ is the Identity matrix and $F=\beta \theta \Delta^w $ with 
$\theta,\beta >0$. If $0< {\theta}<\frac {1}{ \beta \|\Delta^w\|_\infty}$
  then $M$ is a symmetric positive definite matrix. 
\end{lemm}
\begin{proof}
By definition of $\Delta^w$ in \eqref{eq:lapl},  it easily follows that $M=(I + \beta \theta \Delta^w)$
is a  real symmetric matrix. Moreover  in the finite discrete setting the $k$-th component of the product $\Delta^\omega u$ is:
\begin{multline}
\label{ElLapl}
\left (\Delta^\omega u \right )_{k}=-(\alpha_{k}^2+\alpha_{k+1}^2+\eta_{k}^2+\eta_{k+N}^2)u_{k}+\alpha_{k+1}^2u_{k+1}+ \\
\alpha_{k}^2 u_{k-1}+\eta_{k+N}^2u_{k+N}+\eta_{k}^2u_{k-N}
\end{multline}
where 
\begin{equation}
\alpha_k = \psi'_{\mu_\ell}(|u_x^{(h)}|_{i,j}) \ \ \eta_k = \psi'_{\mu_\ell}(|u_y^{(h)}|_{i,j}) \ \ k = (i-1)N+j 
\label{eq:weights}
\end{equation} 
Hence on each row there are at most  five non zero elements given by:\\
\begin{equation}
\label{definB}
\begin{tabular}{ll}
$M_{k,k}$=& ${1-}{\beta \theta}(\alpha_{k}^2+\alpha_{k+1}^2+\eta_{k}^2+\eta_{k+N}^2)$,\\
$M_{k,k-1}$=&${\beta  \theta}\alpha_{k}^2$, \\
$M_{k,k+1}$=&${\beta  \theta}\alpha_{k+1}^2$,\\
$M_{k,k-N}$=&${\beta  \theta}\eta_{k}^2$,\\
$M_{k,k+N}$=&${\beta \theta}\eta_{k+N}^2$\\
\end{tabular}
\end{equation}
In order to guarantee that the matrix $B$ is positive definite, 
it is sufficient to determine $\theta$
such that $M$ is strictly diagonally dominant, namely
\begin{multline*}
\abs {1-  \beta \theta (\alpha_{k}^2+\alpha_{k+1}^2+\eta_{k}^2+\eta_{k+N}^2) }>\beta  \theta
(\alpha_{k}^2+\alpha_{k+1}^2+\beta_{k}^2+\eta_{k+N}^2) \\ \quad \forall \,\, k=1,...,N^2
\end{multline*}
$$ \frac{1}{\beta \theta} > 2 (\alpha_{k}^2+\alpha_{k+1}^2+\eta_{k}^2+\eta_{k+N}^2) \quad \forall \,\, k=1,...,N^2$$ 
It easily follows that this relation is satisfied for 
$$\beta \theta <  \max_{k=1,..N^2}\left ( 2(\alpha_{k}^2+\alpha_{k+1}^2+\eta_{k}^2+\eta_{k+N}^2) \right )$$
namely, 
\begin{equation}
\label{CBDP}
 0< {\theta}<\frac {1}{\beta \|\Delta^w\|_{\infty}}
\end{equation}
\end{proof}
Proof of theorem \ref{th:t1}
\begin{proof}
 Using the Householder-Johns theorem \cite{A,B} $\rho(E^{-1}F)<1$ iff $A$ is symmetric positive definite (SPD) and $E^*+F$ 
 is symmetric and positive definite, where $E^*$ is the conjugate transpose of $E$. The matrix $A$ is SPD since it is symmetric and strictly diagonal dominant.
 From  Lemma \ref{lm:l1}, we have $E^*+F=M$ and therefore the condition on $\theta$ guarantees that $ E^*+F$ 
 is symmetric and positive definite.
\end{proof}
 
%

\end{document}